\documentclass[10pt]{amsart}
\usepackage[all]{xy}
\xyoption{rotate}
\usepackage[utf8]{inputenc}
\usepackage{hyperref}
\usepackage{amssymb}
\usepackage{graphicx}

\newcommand{\vin}{\rotatebox[origin=c]{90}{$\in$}}

\newcommand{\C}{{\mathbb C} }
\newcommand{\R}{{\mathbb R} }
\newcommand{\cA}{{\mathcal A} }

\newcommand{\cE}{{\mathcal E} }

\newcommand{\cM}{{\mathcal M} }

\newcommand{\cO}{{\mathcal O} }

\newcommand{\cT}{{\mathcal T} }

\newcommand{\cX}{{\mathcal X} }

\newcommand{\cZ}{{\mathcal Z} }

\newcommand{\cK}{{\mathcal K} }
\newcommand{\wh}{\widehat}
\newcommand{\wt}{\widetilde}
\newcommand{\pt}{\partial}

\def\ol#1{{\overline{#1}}}
\newtheorem*{Maintheorem*}{Main Theorem}
\newtheorem*{theorem*}{Theorem}

\newtheorem*{mainlemma}{Main Lemma}
\newtheorem{theorem}{Theorem}
\newtheorem{definition}{Definition}
\newtheorem{lemma}{Lemma}
\newtheorem{remark}{Remark}
\newtheorem*{remark*}{Remark}
\newtheorem{proposition}{Proposition}
\newtheorem{corollary}{Corollary}

\newtheorem*{observation*}{Observation}

\newtheorem*{notation*}{Notation}
\newtheorem*{assumption*}{General assumption}
\newtheorem*{empty*}{}

\def\ke{K{\"a}h\-ler-Ein\-stein }
\def\ks{Ko\-dai\-ra-Spen\-cer }
\def\ka{K{\"a}h\-ler }

\def\wp{Weil-Pe\-ters\-son }

\def\tei{Teich\-mül\-ler }

\def\ma{Monge-Am\-père }
\def\ii{\sqrt{-1}}
\def\idb{\sqrt{-1}\partial\overline{\partial} }
\def\C{\mathbb{C}}

\def\cinf{C^\infty}

\def\gab{{g_{\alpha\ol\beta}}}

\def\db{{\ol\partial}}
\def\psh{plurisubharmonic}
\def\RP{$R^{n-p}f_*\Omega^p_{\cX/S}(\cK_{\cX/S})$\ }
\def\we{\wedge}
\def\isom{\stackrel{\sim}{\longrightarrow}}

\makeatletter
\newcommand*\bigcdot{\mathpalette\bigcdot@{.5}}
\newcommand*\bigcdot@[2]{\mathbin{\vcenter{\hbox{\scalebox{#2}{$\m@th#1\bullet$}}}}}
\makeatother

\title[Moduli of canonically polarized manifolds]{Moduli of canonically polarized manifolds, higher order Kodaira-Spencer maps, and an analogy to Calabi-Yau manifolds}
\author[G.~Schumacher]{Georg Schumacher}
\address{Fachbereich Mathematik und Informatik,
Philipps-Universit\"at Marburg, Lahnberge, Hans-Meerwein-Straße, D-35032
Marburg, Germany}
\email{schumac@mathematik.uni-marburg.de}
\begin{document}

\begin{abstract}
Yau's solution of the Calabi conjecture made a differential geometric study of moduli spaces possible. The \wp metric, which is a \ka metric for moduli of canonically polarized manifolds, and for polarized Calabi-Yau manifolds, reflects the variation of the \ke metrics in a holomorphic family. Incidentally the existence of \ke metrics implies an analytic proof for the existence of the corresponding moduli spaces. In order to show that the moduli stack of canonically polarized manifolds is hyperbolic, one has to consider higher order \ks maps. We compute the curvature of the related twisted Hodge sheaves $R^{n-p}f_*\Omega^p_{\cX/S}(\cK_{\cX/S})$ for holomorphic families $f:\cX \to S$. The result exhibits a formal analogy to the classical curvature formula for Hodge bundles for families of  Calabi-Yau manifolds. We construct a Finsler metric of negative curvature on the moduli stack of canonically polarized manifolds, whose curvature is bounded from above by a negative constant. An extra argument together with Demailly's version of the Ahlfors Lemma are needed for those points, where the twisted Hodge sheaves are not locally free.
\end{abstract}

\maketitle

\section{Introduction}
A differential geometric study of moduli spaces was based upon Yau's solution of the Calabi conjecture.  The existence of \ke metrics on canonically polarized and (polarized) Calabi-Yau manifolds permitted a generalization of the classical \wp metric to moduli spaces of manifolds of higher dimension. Incidentally, in the analytic category, \ke metrics allow a direct argument for the existence of moduli spaces.

Our main result states that the moduli space of canonically polarized complex manifolds is Kobayashi hyperbolic.

For Riemann surfaces the \wp metric has been studied extensively, and the hyperbolicity of \tei space was understood from various viewpoints. One way was the identification of the \tei metric and the Kobayashi metric by Royden \cite{royden}. Another is the realization of \tei space as a bounded domain. The curvature of the \wp metric was computed by Wolpert \cite{wo} and Tromba \cite{tr}. In particular, the holomorphic sectional curvature turned out to be bounded by a negative constant, thus also implying hyperbolicity.  At this point one could see that the Weil-Petersson metric satisfies a curvature condition that is stronger than negativity of the sectional curvature (cf.\cite{sch:harm}) --– an even stronger property was later shown by Liu, Sun and Yau in \cite{lsy:teich}. The moduli spaces of Riemann surfaces are not hyperbolic, yet hyperbolic in the orbifold sense.

In higher dimensions the curvature of the generalized \wp metric for families of canonically polarized manifolds was computed by Siu in \cite{siu:canlift}, and a formula only in terms of harmonic \ks forms  was given in \cite{sch:curv}. As opposed to the classical case of families of compact Riemann surfaces a potentially positive further term occurred. In view of the result of Viehweg and Zuo \cite{v-z} on the Brody hyperbolicity of the moduli stack it became apparent that higher cohomology groups had to be included.

In \cite{sch-preprint08} (cf.\ \cite{sch:curv}) we introduced higher order \ks maps for deformations of a compact \ka manifold $X$  equipped with a \ke metric of constant negative Ricci curvature. These are notably different from those maps that arise related to obstruction theory and Massey products.  The maps have values in the spaces $H^p(X,\Lambda^p_X\cT_X)$, which carry natural $L^2$-metrics. The idea was to offset unwanted (possibly positive) terms by negative contributions from the next higher order \ks map. (The highest last
order term was always negative though.) Being defined on the symmetric powers of the tangent spaces of the base, these were used in \cite{sch-preprint08} to define a convex sum of metrics, which amounts to a Finsler metric on the base spaces of universal deformations thus descending to the moduli space in the orbifold sense.

In \cite{sch-preprint10}, for a family $f: \cX \to S$ we computed the curvature of the twisted Hodge bundles \strut \hfill \strut  $R^{n-p}f_*\Omega^p_{\cX/S}(\cK_{\cX/S})$ (cf.\ Theorem~\ref{th:curvgen} below):

\begin{gather*}
\hspace{-5cm}
R(A,\ol A,\psi,\ol\psi)=\int_{X}
\left(\left( \Box  + 1 \right)^{-1}(A\cdot \ol A)\right) \cdot(\psi \cdot \ol
\psi) g\/ dV \\ \hspace{2cm}
 \quad +  \int_{X} \left(\left( \Box  + 1 \right)^{-1} (A\cup\psi)\right)
\cdot (\ol A \cup \ol\psi) g\/ dV
 \quad +  \int_{X} \left(\left( \Box  - 1 \right)^{-1}
(A\cup\ol\psi)\right)\cdot (\ol A \cup \psi) g\/ dV.
\end{gather*}

There are technical reasons to consider these sheaves rather than the dual sheaves $R^pf_*\Lambda^p\cT_{\cX/S}$. However, the remarkable fact is that the curvature satisfies an estimate of the form
$$
(*) \qquad R(A,\ol A, \psi,\ol \psi)\geq \|H(A\cup \psi)\|^2 - \|H(\ol A \cup \psi )\|^2
$$
(cf.\ Corollary~\ref{co:curvest}).

One can see that the right-hand side of the above estimate $(*)$ is formally equal to the curvature for families of polarized Calabi-Yau manifolds being induced from the period map domain by the Torelli map.

The same methods also yield the strict positivity of the relative canonical bundle for effectively parameterized families (cf.\ \cite{sch-preprint08,sch-inv}).

In \cite{to-yeung} To and Yeung resumed our approach of higher \ks mappings and Finsler metrics. They are able to apply a somewhat technical curvature formula for the above bundles, which would also lead to our curvature formula. New is \cite[Lemma 13 (ii)]{to-yeung}, where they apply the Cauchy-Schwarz inequality to the second term of $(*)$. This inequality plugged into the curvature of the Finsler metric yields an upper strictly negative estimate for the curvature, whereas in \cite{sch-preprint10} we had to restrict ourselves to relatively compact subspaces of the moduli space.

In order to be precise, it must be assumed for all of these arguments that the above direct image sheaves are locally free, a point that was left open in \cite{to-yeung} so that hyperbolicity only follows for the subspaces, where the dimension of all induced cohomology groups $H^p(X,\Lambda^p\cT_X)$ is constant.
Let $\cM_N\subsetneq\ldots \subsetneq\cM_{j+1}\subsetneq \cM_{j}\subsetneq\ldots\subsetneq \cM$ be the corresponding stratification of the moduli space $\cM$, where all direct cohomology groups are of constant dimension on the spaces $\cM_j\backslash\cM_{j+1}$.

This difficulty can be overcome (cf.\ \cite{sch-inv,sch-CR})  by estimating the metric along curves that traverse the strata. A way to describe this approach is to say that eventually any such stratum $\cM_j$ is hyperbolic modulo $\cM_{j+1}$, and moreover $\mathcal M_j\backslash \mathcal M_{j+1} \subset \mathcal M_j$ is hyperbolically embedded, which yields hyperbolicity on the whole.

\section{Analytic structure of the moduli space}
Since the work of Mumford moduli of canonically polarized complex varieties have been studied in Algebraic Geometry. An analytic approach to the moduli space of canonically polarized manifolds can be based upon the existence and uniqueness of \ke metrics according to Yau's theorem. This approach is most suitable, if a differential geometric study is intended.

\subsection*{\ke metrics}
Let $X$ be a compact canonically polarized complex manifold of dimension $n$.
By definition the canonical line bundle $\cK_X$ possesses a hermitian metric $\wt h$, whose curvature form $\wt \omega_X = - \ii \pt \ol \pt (\log \wt h )$ is a \ka form. The inverse $\wt h^{-1}$  is interpreted as a volume form $\wt g\, dV$ on $X$ whose curvature is the Ricci form $Ric(\wt g)$.

Yau's theorem \cite{yau} states the existence of  a unique \ke form $\omega_X$ of constant Ricci curvature equal to $-1$. This statement is equivalent to the existence of a unique solution of a \ma equation. Namely, the volume form of $\wt\omega_X$ differs from the given form $\wt g\, dV$ by a differentiable factor
$$
\wt \omega_X^n= e^f \wt g\, dV.
$$
(We use the convention $\eta^k= \eta\we\ldots\we \eta/k!$ for any differential form $\eta$.)
The \ke form $\omega_X$ is cohomologous to $\wt \omega_X$ so that $\omega_X= \wt\omega_X + \ii \pt\ol\pt \varphi$ for a differentiable real valued function $\varphi$ on $X$, and the \ma equation reads
$$
(\wt \omega_X + \ii\pt\ol\pt\varphi)^n = e^{f+ \varphi} \wt g\, dV  \text{\quad i.e. \quad}(\wt \omega_X + \ii\pt\ol\pt\varphi)^n = e^\varphi \wt \omega^n_X
$$
which yields the \ke condition
$$
-Ric(\omega_X)= \omega_X.
$$
Concerning holomorphic families, we fix the notation first. Holomorphic local  coordinates on the given manifold $X$ are denoted by $(z^1,\ldots,z^n)$, and the \ke form on $X$ is written as
$$
\omega_X= \ii g_{\alpha\ol \beta} dz^\alpha\we dz^{\ol \beta}.
$$
We set
$$
g= \det(g_{\alpha\ol \beta})
$$
and denote by $dV$ the Euclidean volume element with respect to these coordinates
so that
$$
\omega_X^n = g \, dV.
$$
A holomorphic family of polarized manifolds $\{\cX_s\}_{S \in S}$ is given by a proper, holomorphic submersion $f: \cX \to S$, with fibers $f^{-1}(s) =\cX_s $ for $s\in S$ -- i.e.\ locally the map is a projection $U\times W \to W$, where $U\subset \C^n$ and $s\in W\subset S$ are open subsets. In general $S$ stands for a not necessarily reduced complex space, which will assumed to be reduced, when dealing with moduli spaces.

We will use the summation convention and denote covariant differentiation on the fibers $(\cX_s,\omega_s)$ with respect to a holomorphic coordinate $z^\alpha$, $\alpha=1,\ldots n$, by $\nabla_\alpha$ or rather use the semi-colon notation $\mbox{\textvisiblespace}_{;\alpha}$. For any tensor $\eta$, say $\eta_\alpha$ we set $\ol\eta_{\ol \beta}:= \ol{\eta_{\beta}}$. If necessary, we use the notation $\pt_\alpha$ or $\mbox{\textvisiblespace}_{|\alpha}$ for ordinary differentiation with respect to $z^\alpha$. Also we will write $\pt_\alpha$ for the coordinate vector field $\pt/\pt z^\alpha$.

Yau's openness argument for the existence of a \ke metric implies the following fact (cf.\ \cite{f-s}).

\begin{proposition}\label{pr:relka}
  Let $\cX \to S$ be a holomorphic family of canonically polarized compact, complex manifolds, where $S$ denotes a reduced complex space. Then the family $\omega_{\cX_s}$ of \ke forms with Ricci curvature equal to $-1$ depends in a $\cinf$ way upon the parameter $s\in S$ giving rise to a relative \ka form $\omega_{\cX/S}$.
\end{proposition}
\begin{proof}
The statement being local, we pick a point $s_0\in S$ and consider the \ke metric $\omega_{\cX_{s_0}}$ on the fiber $\cX_{s_0}$. It gives rise to a hermitian metric on the canonical bundle $\cK_{\cX_{s_0}}$, which we extend to a hermitian metric on $\cK_{\cX/S}$ after replacing $S$ by a neighborhood of $s_0$, if necessary. By continuity the curvature form of the latter metric restricted to all neighboring fibers is strictly positive. We call the relative form $\wt \omega_{\cX/S}$. By assumption the induced form $\wt\omega_{\cX_{s_0}}$ is equal to the \ke form $\omega_{\cX_{s_0}}$. So far, the argument also applies to singular parameter spaces $S$.

Let $X=\cX_{s_0}$. In terms of a differentiable trivialization of the given holomorphic family, any differentiable function $\varphi$ on $X$ gives rise to functions $\varphi_s$ on the fibers $\cX_s$. For sufficiently large $k$ and $0<\lambda<1$ the Hölder spaces $C^{k, \lambda}(X)$ are being used. We first assume that $S$ is smooth. A differentiable map of Banach spaces
$$
\Psi: S\times C^{k+2, \lambda}(X) \to C^{k, \lambda}(X)
$$
is defined by
$$
(s,\varphi) \mapsto \log((\wt\omega_{\cX_s}+ \ii \pt\ol\pt \varphi_s)^n/\wt\omega_{\cX_s}^n  )-\varphi_s.
$$
\ke metrics on the fibers $\cX_s$ correspond to solutions of the equation $\Psi(s,\varphi_s)=0$.

By assumption $\Psi(s_0,0)=0$, and the partial derivative $D_1\Psi$ with respect to $S$ is equal to the invertible operator $-(\Box_{\omega_X} + id)$, where $\Box_{\omega_X}$ is the Laplacian on functions with respect to the \ke metric on the central fiber $X$. Now the implicit function theorem yields the claim.

For a singular base space it is necessary to consider a neighborhood of $s_0$ in a smooth ambient space $W$ of minimal dimension containing $S$ with the following properties. From the construction of a \hbox{(semi-)}universal deformation it is known that there exists a family of almost complex structures on $X$ parameterized by $W$ that are integrable for $s\in S$. In a similar way there exists a complex line bundle over $X\times W$ that is holomorphic and equal to the relative canonical bundle when restricted to $S$. Also the differential operators $\pt$ and $\ol\pt$ extend as operators. These facts imply that the map $\Psi$ can be extended to $W$ (after replacing the space by a neighborhood of $s_0$ if necessary). The extension is not unique, but the implicit function theorem yields a solution $\varphi(s)$ on $S$, which is the {\em restriction of a differentiable function on a smooth ambient space}.
\end{proof}

\begin{proposition}\label{re:omX}
  Let $\omega_{\cX/S}^n$ be the relative volume form for a family of \ke metrics  $\omega_{\cX_s}$. Then the closed, real $(1,1)$-form
\begin{equation}\label{eq:omX}
  \omega_\cX = \idb \log (\omega_{\cX/S}^n)
\end{equation}
on the total space $\cX$, restricted to any fiber $\cX_s$ equals the \ke form $\omega_{\cX_s}$.
\end{proposition}

\begin{proof}
Let $\omega_{\cX_s}= \ii(g_{\alpha\ol \beta}(z,s))dz^\alpha\we dz^{\ol \beta}$, and $\omega_{\cX/S}^n= \det(g_{\alpha\ol \beta}(z,s))\, dV(z,s))$, where $dV$ denotes the relative Euclidean volume form. Then
$$
\omega_\cX|{\cX_s}= \idb\log\det(g_{\alpha\ol \beta}(z,s))|\cX_s= -R_{\gamma \ol \delta}(z,s) dz^\gamma\we dz^{\ol \delta} = \omega_{\cX_s}.
$$
\end{proof}

\subsection*{Notions of deformation theory} In the following sections, we will summarize basic facts and give outlines of proofs (cf.\ \cite{Sch1,sch:alg}).

Given a compact complex manifold $X$ and a complex space $S$ together with a distinguished point $s_0\in S$ a {\em deformation} $\xi$ of $X$ over $(S,s_0)$ is given by
\begin{itemize}
  \item[(i)] A proper, holomorphic submersion $f:\cX \to S$
  \item[(ii)] A biholomorphic map $\varphi: X \stackrel{\sim}{\longrightarrow} \cX_{s_0}$.
\end{itemize}
When using the methods of deformation theory, it will often be necessary to replace the base space by a neighborhood of the distinguished point. In this sense it is meaningful to consider deformations over germs of complex spaces.

Note that an {\em isomorphism} of deformations is given by a biholomorphic map of the total spaces over the base space of the given families $(i)$ that is compatible with the isomorphisms $(ii)$.

If $q: (R,r_0) \to (S,s_0)$ is a holomorphic map of complex spaces with distinguished base points and $\xi$ a deformation of $X$ over $(S,s_0)$ then a deformation $q^*\xi$ of $X$ over $(R,r_0)$ is constructed by {\em base change}\/: The corresponding holomorphic family is $\cX_R :=\cX \times_S R \to R$, and the isomorphism $X \to \cX_{R,r_0}$ is induced by $\varphi$ and the canonical isomorphism $\cX_{s_0} \isom \cX_{R,r_0}$. The deformation $q^*\xi$ is called the {\em pull-back} of $\xi$ with respect to $q$.

A deformation $\xi$ of $X$ is called {\em universal}, if after replacing the base spaces by neighborhoods of the distinguished points, any deformation of $X$ is isomorphic to a pull-back $q^*\xi$ of $\xi$, where the map $q$ (as a holomorphic map of space germs) is uniquely determined. In general, only {\em semi-universal} deformations exist. (Here the condition of the uniqueness of the base change map $q$ is weakened to the uniqueness of the tangent map of $q$ at the distinguished point).

\subsection*{Application to moduli spaces}
In the analytic category moduli spaces are constructed by means of deformation theory. Set theoretically a (coarse) moduli space $\cM$ consists of all isomorphism classes of complex manifolds from a given class. For holomorphic families $f:\cX \to S$ of such manifolds, the natural maps $\varphi_f:S \to \cM$ that send a point $s\in S$ to the isomorphism class of its fiber provide the set $\cM$ with a natural (quotient) topology. A natural equivalence relation $\sim$ is defined on the union of all such base spaces $S$, which identifies points, if the respective fibers are isomorphic so that $\cM=(\cup S)/\! \sim$. Note that on any connected component of the moduli space the underlying differentiable structure of the fibers is fixed so that the set of isomorphism classes of complex structures exists.

\begin{theorem}
  The moduli space $\cM$ of canonically polarized manifolds is a coarse moduli space, i.e.\ $\cM$ possesses a complex structure such that for all holomorphic families $f:\cX \to S$ the map $\varphi_f:S\to \cM$ that sends a point of $S$ to the isomorphism class of its fiber, is holomorphic.
\end{theorem}

It can be seen immediately that universality of semi-universal deformations is required. Since canonically polarized manifolds do not possess non-vanishing holomorphic vector fields, this condition is satisfied. Still, in general, there exist points in the parameter spaces $S$, whose fibers are isomorphic.

We have the following general statement.
\begin{proposition}\label{pr:univdef}
Assume that there exists a universal deformation $\xi$ of a compact complex manifold $X$ over a space $(S,s_0)$. Then
\begin{itemize}
  \item[(i)] there exists a natural action of the group $Aut(X)$ on the space germ of $(S,s_0)$.
  \item[(ii)] The ineffectivity kernel of the above group action consist of those automorphisms that extend to automorphisms of $\cX$ over $S$.
  \end{itemize}
\end{proposition}

The {\em proof} follows from a pure deformation theoretic argument: Let $\alpha\in Aut(X)$. Then in terms of the above notation replacing $\varphi$ by $\varphi\circ \alpha$ we obtain a further deformation of $X$, which is isomorphic to a pull-back $q^*_\alpha(\xi)$ under a holomorphic map $q_\alpha:(S,s_0) \to (S,s_0)$, since $\xi$ is universal by assumption. This proves the first assertion. The second statement follows from the definition. \qed

The existence of unique \ke metrics implies the following key property, which we state here for families of canonically polarized manifolds. For {\em polarized} families of Ricci flat manifolds the analogous statement is true.
\begin{mainlemma}
Let $\cX \to S$ and $\cZ \to T$ be holomorphic families of \ke manifolds with constant Ricci curvature equal to $-1$. Let $s_\nu \in S$ and $t_\nu \in T$ be points converging to $s_0\in S$ and $t_0\in T$ resp. Suppose that there exist biholomorphic maps of the fibers $\psi_\nu:\cX_{s_\nu} \isom \cZ_{t_\nu}$. Then a subsequence converges to an isomorphism $\psi_0 :\cX_{s_0} \isom \cZ_{t_0}$.
\end{mainlemma}

\begin{proof}
It follows from the uniqueness of the \ke metrics on the fibers that the isomorphisms $\psi_\nu$ are isometries. As these are isometries in the sense of metric spaces, a subsequence converges to an isometry, which must be an isometry, also in the sense of \ka manifolds by the theorem of Myers and Steenrod. Hence, the limit fibers are isomorphic as complex manifolds.
\end{proof}

\begin{corollary}\label{co:hausd}
  The topology of the moduli space of canonically polarized manifolds is Hausdorff.
\end{corollary}
We will see below that the Hausdorff property is related to the existence of an analytic structure.

\begin{proof}[Proof of the Corollary]
  Given two different points of $p,q \in \cM$ one can find base spaces $(S,s_0)$ and $(T,t_0)$ of universal deformations so that $p$, and $q$ are the images of $s_0$ and $t_0$ resp. If $p$ an $q$ cannot be separated by open subsets, in any neighborhood $S_\nu$ and $T_\nu$  of $s_0$ in $S$ and $t_0$ in $T$ resp.\ there exist equivalent points $s_\nu$ and $t_\nu$. The Main Lemma would imply that $p=q$ in $\cM$.
\end{proof}

Together with the fact that there are no non-vanishing holomorphic vector fields, the Main Lemma also yields an analytic argument for the fact that the automorphism group of a canonically polarized manifold is finite.

\begin{proposition}\label{pr:S/G}
  Let $(S,s_0)$ be a sufficiently small base space on which $Aut(X)$ acts. Then the canonical map
  $$
  \mu:S/Aut(X) \to S/\!\!\sim
  $$
  is a homeomorphism.
\end{proposition}
\begin{corollary}
  The moduli space $\cM$ carries a natural complex structure.
\end{corollary}
\begin{proof}[Proof of the Corollary]
  The moduli space is glued together from complex spaces of the form $S/\!\!\sim$. The gluing maps can be lifted locally to the base spaces of universal deformations as holomorphic maps, since  the universal families $\cX\to S$ yield universal deformations for any point of the base.  The Hausdorff property was already shown in Corollary~\ref{co:hausd}.
\end{proof}

\begin{proof}[Proof of Proposition~\ref{pr:S/G}]
  After replacing $S$ by a neighborhood of $s_0$ there are no further fibers isomorphic to $X$, since the orbit of $Aut(X)$ is finite. The construction implies that $\mu$ is continuous and open. If it is not injective on any neighborhood of the distinguished point, there exist pairwise different points $s_\nu$ and $t_\nu$ in $S$ converging to $s_0$, whose fibers are pairwise isomorphic. We apply the Main Lemma and see that a subsequence of such isomorphisms converge to an automorphism of the central fiber. Using the action of $Aut(X)$ on $S$ from Proposition~\ref{pr:univdef}, we can assume without loss of generality that the limiting automorphism is the identity. It is known that isomorphisms of holomorphic families are parameterized by a complex space \cite{fu1,Sch1}. In our situation, we have a sequence of points with a limit point in this space of isomorphisms. This means that there exists an analytic curve $C$ with two different holomorphic maps $\alpha, \beta:C \to S$, such that the pull-backs of the given  holomorphic families via $\alpha$ and $\beta$ resp.\ are isomorphic. Moreover, at $s_0$ the isomorphisms yield the identity of $X$ so that also the deformations over $C$ are equal. By universality of $\xi$, the maps $\alpha$ and $\beta$ would have to be equal.
\end{proof}

\section{Properties of the \wp metric}\label{se:pw}
Applying the methods of deformation theory to classical \tei theory A.~Weil suggested to study a hermitian metric on the \tei space that is compatible with base change i.e.\ with the action of the \tei modular group and hence descends to the moduli space of compact Riemann surfaces \cite{weil}. In the context of automorphic forms the inner product had been studied by H.~Petersson earlier. For quadratic holomorphic differentials on Riemann surfaces, a canonical inner product is defined in terms hyperbolic metrics, providing the cotangent bundle of the \tei space with a~hermitian metric. On the tangent bundle the dual metric is defined in a natural way for harmonic Beltrami differentials. The \wp metric was shown to be \ka and of negative Ricci curvature by Ahlfors. The curvature tensor of the classical \wp metric was computed by Wolpert \cite{wo} and Tromba \cite{tr}.

Yau's solution of the Calabi problem allowed the construction of a generalized \wp metric for families of canonically polarized manifolds and polarized Calabi-Yau manifolds, (for details refer to \cite{f-s}.)

We summarize basic properties (cf.\ \cite{Sch1,sch-preprint08,sch-preprint10,sch-inv}). The \ks map for a deformation of a complex manifold was originally given in terms of a differentiable trivialization. This approach can be interpreted as follows: For any holomorphic, proper, submersion $f:\cX \to S$ the exact sequence
\begin{equation}\label{eq:tang}
0 \to \cT_{\cX/S} \to \cT_\cX \to f^* \cT_S \to 0
\end{equation}
is being considered. For any point $s \in S$ the edge homomorphism of the associated exact sequence of direct image sheaves
\begin{equation*}\label{eq:tangks}
f_*  f^* \cT_S \to R^1f_* \cT_{\cX/S}
\end{equation*}
yields the \ks map $\rho_{s}:T_{s}S \to H^1(\cX_{s}, T_{\cX_s})$.
The construction is ''functorial'' i.e.\ compatible with the base change of holomorphic families. Now the computation in terms of Dolbeault cohomology yields the classical definition of the \ks map in the following way.

We first assume that $S$ is smooth.  A differentiable splitting of \eqref{eq:tang} is determined by a differentiable trivialization. So, if $s$ is a coordinate function on $S$, and  $\pt/\pt s\in T_sS$ a tangent vector at a point of $S$, then the splitting yields a  differentiable vector field $\pt/\pt s \, + \, b^\alpha(z,s) \pt/\pt z^\alpha$ on $\cX$ that projects down to $\pt/\pt s$ -- again $(z^1,\ldots, z^n)$ are local coordinates on the fibers of $f$. Now
$$
B(z)^\alpha_{\; \ol \beta} \pt_\alpha dz^{\ol \beta} = \ol\pt(\pt/\pt s \, + \, b^\alpha(z,s) )| \cX_s
$$
is a $\ol\pt$-closed form on the fibers which need not be $\ol\pt$-exact as a form on the fibers.

It is easy to see that the above process only requires calculations on the first infinitesimal neighborhoods of the fibers. Altogether we have the following statement.

\begin{lemma}\label{le:ks}
The \ks map
$$
\rho_{s}:T_{s}S \to H^1(\cX_{s}, T_{\cX_s})
$$
assigns to a tangent vector $\pt/\pt s$ the cohomology class $[B^\alpha_{\; \ol \beta }(z,s)\pt_\alpha dz^{\ol \beta} ]$.
\end{lemma}

Let $\omega_\cX $ be the curvature form of  $(\cK_{\cX/S},h) $ like in Proposition~\ref{re:omX}. Let the base space $S$ be an open subset $W=\{(s_1,\ldots,s_m)\}\subset \C^m$ or more generally let $m$ be the embedding dimension of $S$ at a fixed point and $S$ holomorphically embedded into such a space $W$. We consider the tangent vectors $\pt / \pt s^k$. Now
$$
\omega_\cX = \ii\left(g_{\alpha \ol \beta}dz^\alpha\we dz^{\ol \beta} + g_{\alpha \ol \jmath }dz^\alpha\we ds^{\ol \jmath} + g_{i \ol \beta} ds^i \we dz^{\ol \beta} + g_{i \ol \jmath} ds^i \we ds^{\ol \jmath}\right)
$$

Since $\omega_\cX|\cX_s$ is equal to the \ke form on $\cX_s$, it is meaningful to define a differentiable lift
$$
u_k=\frac{\pt}{\pt s^k} + a^\alpha_k(z,s) \frac{\pt}{\pt z^\alpha}
$$
such that $u_k$ is perpendicular to the fiber $\cX_s$. It is called {\em horizontal lift} of $\pt/\pt s^k$.
\begin{lemma}\label{le:horli}
The horizontal lift of $\pt/\pt s^k$ equals
\begin{equation}\label{eq:horli}
  u_k=\frac{\pt}{\pt s^k} - g^{\ol \beta \alpha}g_{k \ol \beta}\frac{\pt}{\pt z^\alpha},
\end{equation}
i.e.\ $a^\alpha_k= - g^{\ol \beta \alpha}g_{k \ol \beta}$,
and
\begin{equation}\label{eq:Ak}
A_k:= \ol\pt (u_k)|\cX_s =A^\alpha_{k \ol \beta}(z,s)\pt_\alpha dz^\ol \beta
\end{equation}

is a representative of the \ks class $\rho(\pt/\pt s^k)$, where $A^\alpha_{k \ol \beta}= a^\alpha_{k;\ol \beta}$.
\end{lemma}
\begin{proof}
  The first statement follows from a simple calculation. The second claim follows from Lemma~\ref{le:ks} and its preceding discussion.
\end{proof}

\begin{lemma}\label{le:harmks}
  The \ks form $A_k$ is harmonic on $\cX_s$ with respect to the \ke form on $\cX_s$, i.e.\
  \begin{equation}\label{eq:harm}
  \ol\pt^*(A_k)=0.
  \end{equation}
  The induced contravariant tensors satisfy
  \begin{equation}\label{eq:symmA}
    A_{k \ol  \beta \ol \delta}= A_{k \ol \delta \ol \beta}.
  \end{equation}
\end{lemma}
\begin{proof}
  We set $g(z,s) = \det(g_{\sigma \ol \tau})$ so that $\omega_\cX = \idb \log g(z,s)$. By \eqref{eq:omX}, and Lemma~\ref{le:horli}
\begin{eqnarray*}
  \ol\pt^*\! (A_k)&=& - g^{\ol \beta \gamma}A^\alpha_{k \ol \beta; \gamma } \pt_\alpha =  g^{\ol \beta \gamma} g^{\ol \delta \alpha} g_{k \ol \delta; \ol \beta\gamma} = g^{\ol \beta \gamma} g^{\ol \delta \alpha}\left( g_{k \ol \beta; \gamma \ol \delta} - g_{k \ol \tau} R^\ol\tau_{\;\ol \beta\ol \delta\gamma}   \right)\\ &=&
  g^{\ol \delta\alpha}\left(\left(\frac{\pt}{\pt s^k}   \log \det(g_{\gamma\ol \beta})\right)_{;\ol \delta }  + g_{k\ol \tau} R^\ol\tau_{\; \ol \delta} \right)=g^{\ol \delta\alpha}( g_{k\ol \delta}-g_{k\ol \delta})=0
\end{eqnarray*}
The symmetry of $A_{k\ol \beta\ol \delta}$ follows from \eqref{eq:horli} and \eqref{eq:Ak}.
\end{proof}
We saw that the harmonicity of $A_k$ can be interpreted as  infinitesimal \ke condition.

The \ke metrics on the fibers of a holomorphic family over a spaces $S$ define a hermitian metric for harmonic \ks forms, which induces semi-positive hermitian inner products on the tangent spaces of the base $S$. It follows from the construction that these inner products are compatible with base change and strictly positive for a universal deformation.

As above $s$ stands for a coordinate function of a parameter space, and $\pt/\pt s$ for a general tangent vector.
\begin{definition}
  Let $\cX \to S$ be a holomorphic family of canonically polarized complex manifolds. Let
  $$
  \rho_s :T_sS \to H^1(\cX_s, \cT_{\cX_s})
  $$
  be the \ks mapping, and let
  $$
  A^\alpha_{s \ol \beta}(z)\pt_\alpha dz^{\ol \beta}
  $$
  be the harmonic representative of $\rho_s(\pt/\pt s)$. Then the \wp norm is defined to be
  $$
  \|\pt / \pt s\|^2_{WP} = \int_{\cX_s} \|A_s\|^2(z) g\, dV =  \int_{\cX_s} A^\alpha_{k \ol \beta} \ol A^{\ol \delta}_{\ol \ell \gamma} g_{\alpha\ol \delta}g^{\ol \beta\gamma} g \, dV = \int_{\cX_s} A^ \alpha_{k \ol \beta} \ol A^{\ol \beta}_{\ol \ell \alpha} g \, dV.
  $$
\end{definition}
(The last equality in the above definition follows from \eqref{eq:symmA}.)

We will use the technique of fiber integration of differential forms. Given a smooth proper holomorphic map $f:\cX \to S$ of complex manifolds with fibers of complex dimension $n$, and $\eta$ a $\cinf$ differential form of degree $2n+r$, the fiber integral
$$
\int_{\cX/S} \eta
$$
is a differential form of degree $r$ and class $\cinf$. The extension of fiber integration to currents is the push-forward.

Again, this construction is compatible with base change. For trivial (differentiable) families the explicit construction is obvious, and one can use a differentiable trivialization otherwise.

The construction is type preserving: If $\eta$ is of type $(n+r,n+s)$, then the fiber integral is of type $(r,s)$. Furthermore, it commutes with the exterior derivatives $d$, $\pt$, and $\ol\pt$.

For smooth families over reduced base spaces the forms $\eta$ are assumed to be locally extendable as $\cinf$ forms to smooth ambient spaces: Locally $f$ is the projection $U \times S \to S$, where $U$ is smooth, and $S\subset W$, where $W$ is an open subset of a complex number space. In this way, the fiber integral of $\eta$ can be provided with a local $\cinf$ extension to a smooth space.
\begin{theorem}
  The \wp hermitian inner product on the tangent spaces defines a \ka form $\omega^{WP}$ on the base spaces of universal deformations.  The \ka form $\omega^{WP}$ is invariant under the action of the automorphism group of the distinguished fiber, and thus descends to the moduli space.

  Locally on the bases spaces of deformations the form $\omega^{WP}$ possesses a $\pt\ol\pt$-potential that is of class $\cinf$ on a smooth ambient space. Such a potential descends to the moduli space as a plurisubharmonic continuous function.
\end{theorem}
The notion of a \ks map is meaningful for families over non-reduced base spaces, and non-reducedness may also occur for universal deformations. In principle the statement of Proposition~\ref{pr:relka} can also be given a meaning for non-reduced base spaces. However, moduli spaces are usually provided with a reduced complex structure. In this sense, we restrict ourselves to reduced parameter spaces at the expense that the \wp inner product may be defined on larger tangent spaces.

\begin{proof}[Proof of the Theorem]
  Since \ke metrics of constant Ricci curvature equal to $-1$ are unique and compatible with biholomorphic mappings, it follows immediately form the definition that the \wp metric is invariant under the action of the automorphism group of the central fiber in a (local) deformation and descends to the moduli space.

  The second part of the Theorem follows from the fiber integral formula \eqref{eq:fibint} below. Note that $\omega_\cX$  possesses local $\pt \ol \pt$-potentials of class $\cinf$ that extend to local, smooth ambient spaces in a $\cinf$ way -- a property that  also holds for $\omega^{WP}$ because of \eqref{eq:fibint}.
  \end{proof}

Given a universal deformation $\cX \to S$, we denote by ${\omega_\cX}$ the form as in \eqref{eq:omX} of Proposition~\ref{re:omX}.

\begin{theorem}[\cite{f-s}]\label{th:fibint}
  The \wp form is given by a fiber integral
  \begin{equation}\label{eq:fibint}
    \omega^{WP}= \int_{\cX/S} \omega^{n+1}_\cX.
  \end{equation}
\end{theorem}
  Because of the base change property for $\omega^{WP}$ and $\omega_\cX$, it is sufficient to prove \eqref{eq:fibint} for spaces of dimension one, when dealing with smooth base spaces. Since we need to include singular base spaces we reduce the statement to embedding dimension one.

  Let $s$ be a coordinate function for $S$, and we use $s$ also as index for this function and for the restriction of the form $\omega_\cX$, whose positive definiteness is to be shown. The following facts can be verified by an obvious calculation.
\begin{lemma}
Let $u_s$ be the horizontal lift of $\pt/\pt s$, and let
$$
\varphi=\varphi_{s \ol s}= \langle u_s ,u_s \rangle_{\omega_\cX} .
$$
Then
\begin{equation}\label{eq:phi}
  \varphi_{s\ol s}= g_{s\ol s}- g^{\ol \beta \alpha}g_{s \ol \beta}g_{\alpha \ol s},
\end{equation}
and $\varphi_{s\ol s}$ is related to the determinant of an $(n+1)\times (n+1)$-matrix:
\begin{equation}\label{eq:det}
  \varphi_{s\ol s}\cdot \det(g_{\alpha\ol \beta}) = \det
\left(
\begin{array}{cc}
g_{s\ol s} & g_{s\ol\beta}\\ g_{\alpha\ol s}& \gab
\end{array}
\right).
\end{equation}
\end{lemma}
A crucial property is the following differential equation. In this way it is possible to eliminate second order derivatives with respect to the parameter space from $\varphi$.
\begin{proposition}\label{pr:laplphi}
Denote by $\Box_s=\Box_{\cX_s}$ the Laplacian on fibers with respect to $\omega_{\cX_s}$. Then
  \begin{equation}\label{eq:laplphi}
   (1 + \Box_s)(\varphi_{s\ol s}(z,s)) = \|A_s\|^2(z,s).
  \end{equation}
\end{proposition}
\begin{proof}
  The calculation depends upon the \ke condition. Details are in \cite{sch:curv,sch-inv}.
\end{proof}
\begin{proof}[Proof of Theorem~\ref{th:fibint}]
  By Proposition~\ref{pr:laplphi} we have
  $$
  \|\pt/\pt s\|^2_{WP}= \int_{\cX_s} (1 + \Box_s)(\varphi_{s\ol s}) g\, dV =\int_{\cX_s} \varphi_{s\ol s}\, g\, dV.
  $$
  Because of \eqref{eq:det} this quantity corresponds to the fiber integral from \eqref{eq:fibint} evaluated at $\pt/\pt s$.
\end{proof}
With $\varphi_{i\ol\jmath}= \langle u_i,u_j\rangle_{\omega_\cX}$ the equation
\begin{equation}\label{eq:oxnpl1}
  \omega^{n+1}_\cX=\ii\, \varphi_{i\ol\jmath}\,ds^i\we ds^{\ol\jmath}\, g\, dV
\end{equation}
follows.

A holomorphic family is called {\em effectively parameterized}, if the \ks mapping is everywhere injective. In particular, universal families have this property.

An important consequence of Proposition~\ref{pr:laplphi} is the following fact.

\begin{theorem}[\cite{sch-preprint08,sch-inv}]
  Let $\cX \to S$ be an effectively parameterized family of canonically polarized manifolds. Endow the relative canonical bundle $\cK_{\cX/S}$ with the hermitian metric $h$ that is induced by the family of \ke volume forms. Then the bundle $(\cK_{\cX/S},h)$ has a strictly positive curvature form $\omega_\cX$, whose restrictions to the fibers $\cX_s$ are equal to the \ke forms $\cX_s$.
\end{theorem}

\begin{proof}
In view of \eqref{eq:oxnpl1} it is sufficient to show that $\varphi_{i\ol\jmath}$ is positive definite on all of $\cX$. We pick a non vanishing tangent vector on $S$ and use the above notation $\pt/\pt s$. Then we need to show that the function $\varphi_{s\ol s}$ in the sense of Lemma~\ref{le:horli} is strictly positive on $\cX$. Since $\|A_s\|^2(z,s)$ is larger or equal to zero on $\cX$ and not identically zero on any fiber, the strict positivity of $\varphi_{s\ol s}$ follows from Proposition~\ref{pr:laplphi} and general theory.
\end{proof}
Based upon the estimates of the heat kernel for the Laplacian by Cheeger and Yau \cite{c-y}, lower estimates for $\varphi_{s\ol s}$ were given in \cite{sch-inv}.

Now the fiber integral \eqref{eq:fibint} together with \eqref{eq:oxnpl1} imply that the \wp form can be defined in terms of horizontal lifts.

\begin{proposition}
  Let $u_i$ be the horizontal lift of $\pt/\pt s^i$, and $\omega_\cX = \ii \pt \ol \pt \log h$  the \ka form on the total space induced by the \ke forms on the fibers. Then
  \begin{equation}
    \omega_{WP} = \left(\int_{\cX_s}\langle u_i, u_j\rangle_{\omega_\cX}  g\, dV\right)\ii ds^i\we ds^{\ol \jmath}
  \end{equation}
\end{proposition}

\section{General properties of direct images}
For any proper holomorphic map $f:\cX \to S$, and any coherent sheaf $\cE$ on $\cX$ that is $\cO_S$-flat, and $q \geq 0$,  the canonical morphism
$$
R^{q}f_*\cE \otimes \C(s) \to H^q(\cX_s, \cE_s)
$$
is being considered, where $\C(s)= \cO_S/ \mathfrak{m}_s $ denotes the sheaf $\C$ concentrated at a point $s\in S$, and $\cE_s=\cE \otimes \cO_{\cX_s}$.
\begin{proposition}
Let $f:\cX \to S$ be a holomorphic family of canonically polarized manifolds over a complex space. Then for all $s\in S$ and $0\leq p \leq n$ the canonical map
\begin{equation}\label{eq:basechange}
R^{n-p}f_* (\Omega^p_{\cX/S}(\cK_{\cX/S})) \otimes \C(s) \to H^{n-p}(\cX_s,\Omega^p_{\cX_s}(\cK_{\cX_s}))
\end{equation}
is an isomorphism.
\end{proposition}
\begin{proof}
According to the Kodaira-Nakano vanishing theorem all groups $H^{n-p+1}(\cX_s, \Omega^p_{\cX_s}(\cK_{\cX_s}))$ vanish. Furthermore the sheaves $\Omega^p_{\cX/S}(\cK_{\cX/S}) $ are $S$-flat. Now Grauert's base change and comparison theorem (cf.\ \cite[Theorem III 3.4 and Corollary 3.5]{banica}) implies the statement.
\end{proof}

Given a fixed number $0\leq p \leq n$, we consider the subspace  $S'\subset S$ of points $s$, where the dimension $h^{n-p}(\cX_s,\Omega^p_{\cX_s}(\cK_{\cX_s}))$ is minimal. Due to the semi-continuity of the dimension of the cohomology groups, $S'$ is the complement of a nowhere dense analytic subset. Let $\cX'= \cX\times_S S'$, and let $f': \cX' \to S'$ be the induced map. Furthermore denote by $S_0$ the reduction of $S'$, set $\cX_0= \cX\times_S S_0$, and denote by $f_0:\cX_0 \to S_0$ the induced map. Now by the base change theorem the pull back $R^{n-p}f_* (\Omega^p_{\cX/S}(\cK_{\cX/S}))\otimes_{\cO_S} \cO_{S_0}$ is equal to $R^{n-p}f_{0*} (\Omega^p_{\cX_0/S_0}(\cK_{\cX_0/S_0}))$.

\begin{remark}\label{re:basech}
Over $S_0$ the sheaf $R^{n-p}f_{0*} (\Omega^p_{\cX_0/S_0}(\cK_{\cX_0/S_0}))$ is locally free, in particular also $R^pf_{0*}\Lambda^p\cT_{\cX_0/S_0}$ is, and the base change property holds for the latter sheaf. In particular the map
$$
R^pf_{0*}(\Lambda^p\cT_{\cX_0/S_0})\otimes_{\cO_{S_0}} \C(s) \to H^p(\cX_s, \Lambda^p\cT_{\cX_s})
$$
is an isomorphism.
\end{remark}
\begin{proof}
We only use the fact that for reduced base spaces the constancy of the dimension of the cohomology groups implies both the local freeness of the direct image sheaf and the base change property.
\end{proof}

\begin{remark}
Let $f:\cX \to S$ be a local universal family such that $\dim H^1(\cX_s,\cT_{\cX_s})$ is constant. Then either  $S$ is smooth or everywhere non-reduced.
\end{remark}
Catanese showed that the latter situation actually occurs for certain surfaces of general type \cite{catanese}.

Sections of a direct image sheaf and families of harmonic representatives of the corresponding cohomology groups are related as follows. Here we assume that the direct image is locally free and restrict the situation to a Stein contractible base space say.

\begin{proposition}\label{pr:harmsec}
Let $f:\cX \to S$ be a smooth proper holomorphic map over a reduced space $S$ with a relative \ka form $\omega_{\cX/S}$ and let $(\mathcal E,h)$ be a locally free, coherent sheaf on $\cX$  with a hermitian metric such that $R^qf_*\cE$ is free. As above $\cE_s$ stands for the analytic restriction $\cE \otimes \cO_{\cX_s}$ of $\cE$ to a fiber $\cX_s$.

Then any section of $R^qf_*\cE(S)$ can be represented by a $\ol\pt$-closed form $\psi\in \cA^{(0,q)}(\cX, \cE)$ whose restrictions $\psi_s \in H^q(\cX_s, \cE_s)$ to the fibers $\cX_s$ are harmonic.
\end{proposition}
The simple argument is given in \cite{sch-inv}.

Denote by $\cE^{\vee}= Hom(\cO_{\cX},\cE)$ the dual sheaf. Let $\cE$  and $\cE^\vee \otimes \cK_{\cX/S}$ satisfy the assumptions of the proposition for $p$ and $n-p$ resp.

\begin{corollary}\label{co:serre}
The relative Serre duality
$$
R^pf_*\cE \otimes_{\cO_S} R^{n-p}f_*(\cE^\vee \otimes_{\cO_\cX}\cK_{\cX/S}) \to \cO_S
$$
can be evaluated in terms of the fiber integral
$$
\int_{\cX/S} \psi\we\chi,
$$
where $\psi \in \cA^{(0,p)}(\cX, \cE)$ and $\chi \in \cA^{(0,n-p)}(\cX,\cE^\vee \otimes \Omega^n_{\cX/S})$ are $\ol\pt$-closed forms that are harmonic, when restricted to fibers.
\end{corollary}
\begin{proof}
The fiber integral defines a function on $S$. Fiberwise we have the classical statement of Serre duality. Since fiber integration commutes with exterior derivatives, the result is a $\ol\pt$-closed, i.e.\ holomorphic function.
\end{proof}

\section{Curvature of higher direct image sheaves}

\begin{assumption*}
From now on we denote by $f:\cX \to S$ a holomorphic family over a reduced complex space such that all direct image sheaves $R^{n-p}f_* (\Omega^p_{\cX/S}(\cK_{\cX/S}))$ are locally free for all $p$. We will return to the general case later.
\end{assumption*}

We mention an identity for harmonic \ks forms $A$ that arises from the symmetry: of $A_{\ol \beta\ol \delta}$.
\begin{equation}
  (A^*)^\ol{\beta}_{\; \alpha}= \ol A^\ol{\beta}_{\;\alpha}
\end{equation}
Namely, by definition $(A^*)^\ol{\beta}_{\;\alpha}= \ol A^\ol\delta_{\;\gamma}\,g^{\ol \beta \gamma}g_{\alpha\ol \delta}= \ol A_{\alpha\gamma}g^{\ol \beta \alpha}$. By \eqref{eq:symmA} this equals $\ol A^\ol{\beta}_{\;\alpha}$.

\begin{definition}\label{de:cup}
Let $s\in S$, and let $A=A^\alpha_{\;\ol\beta}(z,s) \pt_\alpha dz^{\ol\beta}$ be a harmonic \ks form on a fiber $X=\cX_s$. Then the cup product together with the contraction defines
\begin{eqnarray}
A\cup \mbox{\textvisiblespace} :
\cA^{0,n-p}(X,\Omega^p_{X}(\cK_{X})) &\to&
\cA^{0,n-p+1}(X,\Omega^{p-1}_{X}(\cK_{X}))\label{eq:cup1}\\
\ol A\cup \mbox{\textvisiblespace} :
\cA^{0,n-p+1}(X,\Omega^{p-1}_{X}(\cK_{X})) &\to&
\cA^{0,n-p}(X,\Omega^{p}_{X}(\cK_{X})),\label{eq:cup2}
\end{eqnarray}
where $0<p\leq n$, namely
\begin{gather*}
  \left(A^\gamma_{\;\ol \delta}\pt_\gamma dz^{\ol \delta}\right) \cup \left( \psi_{\alpha_1,\ldots,\alpha_p, \ol \beta_{p+1},\ldots, \beta_n}dz^{\alpha_1}\we \ldots \we dz^{\alpha_p}\we dz^{\ol \beta_{p+1}}\we \ldots \we dz^{\ol \beta_n}\right) = \hspace{4cm}\\ \hspace{3cm}A^\gamma_{\;\ol \beta_p} \psi _{\gamma \alpha_1,\ldots,\alpha_{p-1}\ol \beta_{p+1},\ldots,\ol
   \beta_n} dz^{\ol \beta_p}\we dz^{\alpha_1}\we \ldots \we dz^{\alpha_{p-1}}\we dz^{\ol \beta_{p+1}}\we \ldots \we dz^{\ol\beta_n} \\
   \left(\ol A^{\ol \delta}_{\; \gamma}\pt_\gamma dz^{\ol \delta}\right) \cup \left( \psi_{\alpha_1,\ldots,\alpha_{p-1},\ol \beta_{p},\ldots, \ol\beta_n}dz^{\alpha_1}\we \ldots \we dz^{\alpha_{p-1}}\we dz^{\ol \beta_{p}}\we \ldots \we dz^{\ol \beta_n}\right)= \hspace{4cm}\\ \hspace{3cm} \ol A^{\ol \delta}_{\; \alpha_1} \psi_{\alpha_2,\ldots, \alpha_{p}, \ol \delta, \ol \beta_{p+1},\ldots,\ol \beta_n} dz^{\alpha_1}\we\ldots\we dz^{\alpha_{p}}\we dz^{\ol \beta_{p+1}}\we \ldots \we dz^{\ol \beta_n}.
\end{gather*}
\end{definition}
\begin{observation*}
 The cup product \eqref{eq:cup1} with the harmonic \ks form $A$ gives rise to cup products for cohomology classes ($\psi$ is assumed to be harmonic).
 \begin{eqnarray}
 A\cup\mbox{\textvisiblespace}&:& H^{n-p}(X, \Omega^p_X(\cK_X))  \longrightarrow
 H^{n-p+1}(X, \Omega^{p-1}_X(\cK_X))\quad  ; \quad  [\psi] \mapsto [A\cup \psi] = [H(A\cup \psi)] \label{eq:cup1a} \\ \ol A\cup\mbox{\textvisiblespace}&:& H^{n-p+1}(X, \Omega^{p-1}_X(\cK_X))  \longrightarrow
 H^{n-p}(X, \Omega^{p}_X(\cK_X))\quad  ; \quad  [\psi] \mapsto [H(\ol A\cup \psi)],\label{eq:cup2a}
 \end{eqnarray}
 where $H$ denotes the harmonic projection. In  \eqref{eq:cup2a} it is necessary to apply the harmonic projection.
\end{observation*}

\begin{lemma}\label{le:dual}
\begin{itemize}
\item[(i)]
  The map $\ol A\cup \mbox{\textvisiblespace}$ from \eqref{eq:cup2} is {\em adjoint} to the map $A\cup \mbox{\textvisiblespace}$ from \eqref{eq:cup1}.
\item[(ii)]
  The map $\ol A\cup \mbox{\textvisiblespace}$ from \eqref{eq:cup2a} is {\em adjoint} to the map $A\cup \mbox{\textvisiblespace}$ from \eqref{eq:cup1a}.
\end{itemize}
\end{lemma}
\begin{proof}
A simple calculation depending on the symmetry \eqref{eq:symmA} shows $(i)$. The $L^2$-inner product on the cohomology is defined in terms of harmonic representatives.  For any harmonic form $\chi \in \cA^{(p,n-p)}(X, \cK_X)$ we have $(H(\ol A\cup \psi), \chi)=(\ol A\cup \psi, \chi) = (\psi, A\cup \chi) = (\psi, H(A\cup \chi))$ which implies $(ii)$.
\end{proof}

The family of \ke metrics for the fibers of $f:\cX \to S$ defines natural hermitian metrics on the fibers of $R^{n-p}f_*\Omega^p_{\cX/S}(\cK_{\cX/S})$, namely the induced $L^2$-inner products for harmonic representatives of cohomology classes from $H^{n-p}(\cX_s, \Omega^p_{\cX_s})$.
Let $\psi^k, \psi^\ell \in \cA^{(0,n-p)}(\cX_s,\Omega^p_{\cX_s}(\cK_{\cX_s}))$ be harmonic, then
$$
\langle \psi^k,\psi^\ell\rangle := \int_{\cX_s} \psi^k_{A_p\ol B_q} \psi^\ol\ell_{\ol D_p C_{n-p}} g^{\ol D_p A_p} g^{\ol B_{n-p} C_{n-p} } dV,
$$
where $A_p = (\alpha_1,\ldots,\alpha_p)$ etc.\ and $dz^{A_p}= dz^{\alpha_1}\we\ldots\we dz^{\alpha_p}$ etc. If the coefficients $\psi^k_{A_p\ol B_{n-p}}$ etc.\ are already skew-symmetric, then $g^{\ol D_p A_p} = g^{\ol \delta_1 \alpha_1 }\cdot \ldots \cdot g^{\ol \delta_p \alpha_p }$ can be taken, otherwise this quantity denotes a determinant.

\subsection*{Computation of the curvature}

\begin{theorem}[{\cite[Theorem IV]{sch-preprint10}}]\label{th:curvgen}
Let $X$ denote a fiber $\cX_s$ at a point $s\in S$. Let the harmonic \ks form $A$ stand for a tangent vector of $S$ at $s$, and represent a vector of the fiber $R^{n-p}f_*\Omega^p_{\cX/S}(\cK_{\cX/S})\otimes \C(s)=H^{n-p}(X, \cE)$ by the harmonic form $\psi$.
\begin{itemize}
  \item [(i)]
The curvature tensor for $R^{n-p}f_*\Omega^p_{\cX/S}(\cK_{\cX/S})$ at $s$ is given by
\begin{eqnarray}\label{eq:curvgen}
R(A,\ol A,\psi,\ol\psi)&=& \int_{X}
\left(\left( \Box  + 1 \right)^{-1}(A\cdot \ol A)\right) \cdot(\psi \cdot \ol
\psi) g\/ dV\nonumber\\
&& \quad +  \int_{X}\left( \left( \Box  + 1 \right)^{-1} (A\cup\psi) \right)
\cdot (\ol A \cup \ol\psi) g\/ dV \\
&& \quad +  \int_{X} \left(\left( \Box  - 1 \right)^{-1}
(A\cup\ol\psi)\right)\cdot (\ol A \cup \psi) g\/ dV.
\nonumber
\end{eqnarray}
The general curvature formula can be derived by polarization.
All non-zero eigenvalues of Laplacians occurring in the third term are greater than one.
\item [(ii)]
The only contribution in \eqref{eq:curvgen}, which may be negative,
originates from the harmonic projections $H(A\cup \ol \psi)$ in the third term. It equals
$$
- \int_{X} H(A\cup\ol\psi) \ol{H(A\cup\ol\psi)} g dV= -\|H(\ol A\cup \psi)\|^2.
$$
\end{itemize}
\end{theorem}
For $p=0$ the second term in \eqref{eq:curvgen} is equal to zero, and for $p=n$ the third one does not occur.

Part $(ii)$ of the theorem is a direct consequence of the first statement. It follows from an eigenvector expansion of $A\cup \ol \psi$.

For $p=n$ the Kodaira-Nakano vanishing theorem together with the Grauert comparison theorem imply that $f_*{\cK^{\otimes 2}_{\cX/S}}$ is locally free for any base space.
\begin{corollary}
  The sheaf $f_*{\cK^{\otimes 2}_{\cX/S}}$ is Nakano positive.
\end{corollary}
\begin{proof}
  The second term of \eqref{eq:curvgen} gives Nakano-positivity immediately. For an estimate one can use the first term using the Cheeger-Yau estimates for the heat kernel \cite{c-y}.
\end{proof}

We read an immediate estimate off \eqref{eq:curvgen} using the second and third term.
\begin{corollary}\label{co:curvest}
\strut\vspace{-5pt}
\begin{equation}\label{eq:curvest}
R(A,\ol A, \psi,\ol \psi)\geq  \|H(A\cup \psi)\|^2 - \|H(\ol A \cup \psi )\|^2.
\end{equation}
\end{corollary}
We will discuss the proof of Theorem~\ref{th:curvgen} in Section~\ref{se:proofcurv}.

\begin{corollary}
  The statement of Theorem~\ref{th:curvgen} holds for any family of canonically polarized manifolds over a reduced space, if the twisted Hodge sheaves $R^{n-p}f_* \Omega^p_{\cX/S}(\cK_{\cX/S})$ are locally free.
\end{corollary}
The {\em proof} follows from the base change theorem (cf.\ Remark~\ref{re:basech}). \qed

We will use it for subspaces of base spaces of universal deformations.

\section{Variation of Hodge structure for families of Calabi-Yau manifolds -- an analogy}
In a geometric situation of a proper, holomorphic family $f:\cX \to S$ we consider the Hodge bundles $E^q= R^qf_*\Omega^{n-q}_{\cX/s}$. These carry  connections arising from the flat connection on $R^n f_*\C$. In fact these are induced by the family of \ke metrics on the fibers. As above we denote by $A$ a harmonic \ks form representing a tangent vector of a point $s\in S$. Again we set $X=\cX_s$. The cup product with $A$ determines a map
\begin{equation}\label{eq:cup1c}
A\cup \mbox{\textvisiblespace} : H^{n-p}(X,\Omega^p_X) \to H^{n-p+1}(X,\Omega^{p-1}_X).
\end{equation}

\begin{theorem}[{Griffiths {\cite[(5.2)]{gr:pe}}}]
The curvature tensor $R_h$ of a Hodge bundle $E^q= R^qf_*\Omega^{n-q}_{\cX/S}$ at a point $s\in S$ is given by
\begin{equation}\label{eq:curvHo}
R_h(A,\ol A, \psi, \ol \psi) = \|H(A\cup \psi)\|^2  - \|H((A\cup)^t\psi)\|^2,
\end{equation}
where $(A\cup)^t: H^{n-p}(X, \Omega^p_{X})\to H^{n-p-1}(X, \Omega^{p+1}_{X})$ denotes the adjoint map with respect to the $L^2$ inner product.
\end{theorem}

It is a classical result that the above computation of the curvature of period domains and the validity of a Torelli theorem imply the hyperbolicity of the moduli space of polarized Calabi-Yau manifolds.

\begin{remark}\label{re:cup}
Like in Lemma~\ref{le:dual} one can see that also for Calabi-Yau manifolds the adjoint $(A\cup)^t\mbox{\textvisiblespace}$ on the level of cohomology is the cup product with $\ol A$. Now \eqref{eq:curvHo} reads
\begin{equation}
  R_{h}(A,\ol A, \psi, \ol \psi) = \|H(A\cup \psi)\|^2  - \|H(\ol A\cup\psi)\|^2.
\end{equation}
\end{remark}
This, and the application to the sheaves $R^pf_*\Lambda^p\cT_{\cX/S}$ was a motivation to compute the curvature of twisted Hodge bundles in \cite{sch-preprint10}.

\section{Higher order \ks maps}
Let $f:\cX \to S$ be a holomorphic family of canonically polarized varieties $\cX_s$, $s\in S$ of complex dimension $n$. Denote by $\rho : T_sS \to H^1(\cX_s, \cT_{\cX_s})$ the \ks map. Let $u_i\in T_{S,s}$ be tangent vectors and $A_i$ (harmonic) representatives of $\rho(u_i)$.
\begin{definition}\label{de:genks}
The generalized \ks maps are defined on the symmetric powers of the tangent spaces.
\begin{eqnarray}
  \rho^{(p)}:S^p(T_{S,s})  &\to & H^p(\cX_s,\Lambda^p\cT_{\cX_s})\enspace \\ \nonumber
  \vin \hspace{5mm} & & \hspace{5mm}\vin \\
  \enspace u_1\otimes\ldots\otimes u_p  &\mapsto & [A_1\we\ldots\we A_p]=[A^{\alpha_1}_{1\ol \beta_1}\pt_{\alpha_1}dz^{\ol{\beta}_1} \we \ldots \we A^{\alpha_p}_{p\ol \beta_p}\pt_{\alpha_p}dz^{\ol{\beta}_p}] . \nonumber
\end{eqnarray}

\end{definition}
The generalized \ks maps are induced by the classical \ks map together with the natural morphisms $S^pH^1(\cX_s,\cT_{\cX_s}) \to  H^p(\cX_s,\Lambda^p\cT_{\cX_s})$.
\begin{remark}
  The \ke metrics on the fibers induce \wp metrics on the symmetric spaces $S^pH^1(\cX_s,\cT_{\cX_s})$, namely
  $$
  \|u_1\otimes \ldots\otimes u_p\|^2_{WP} := \|H(A_1\we\ldots A_p)\|^2.
  $$
\end{remark}
\begin{notation*}
For any harmonic \ks form $A\in \cA^{(0,1)}(X,\cT_X)$ we define
\begin{equation}\label{eq:Ap}
A^p :=H(A\we\ldots\we A).
\end{equation}
\end{notation*}

\subsection*{Curvature formula}
The cup products \eqref{eq:cup1} and \eqref{eq:cup2} from Definition~\ref{de:cup} yield wedge products on the dual spaces.
\begin{eqnarray}
A\we \mbox{\textvisiblespace} :
&&\cA^{0,p}(\cX_s,\Lambda^p\cT_{\cX_s}) \to
\cA^{0,p+1}(\cX_s,\Lambda^{p+1}\cT_{\cX_s})\label{eq:we1}\\ \ol A \we \mbox{\textvisiblespace} :
&&\cA^{0,p+1}(\cX_s,\Lambda^{p+1}\cT_{\cX_s}) \to
\cA^{0,p}(\cX_s,\Lambda^{p}\cT_{\cX_s}).\label{eq:we2}
\end{eqnarray}
Here, in  \eqref{eq:we1} the wedge product stands for an exterior product,
whereas in \eqref{eq:we2} the wedge product contains a contraction with
both indices of $A_i$. Again for $p=n$ in \eqref{eq:we1} and \eqref{eq:we2} the value of the wedge product is zero.

Lemma~\ref{le:dual} implies
\begin{lemma}\label{le:dualdual}
The map \eqref{eq:we2} is dual to \eqref{eq:we1}.
\end{lemma}

We apply our result to the curvature for twisted Hodge bundles, again assuming that the dimension of the cohomology groups (for all degrees $p$) is constant. Observe that the role of conjugate and non-conjugate \ks forms is interchanged.
\begin{theorem}[\cite{sch-preprint10}]\label{th:curvgendual}
Let $X=\cX_s$. Let $A\in \cA^{(0,1)}(X, \cT_X)$ be a harmonic \ks form representing a tangent vector of $S$ at $s$, and let $\nu \in \cA^{(0,p)}(X,\Lambda^p\cT_X)$ be a harmonic form.

Then the curvature of $R^pf_*\Lambda^p\cT_{\cX/S}$ at $s$ equals
\begin{eqnarray}\label{eq:curvgendual}
R(A,\ol A, \nu, \ol \nu )&=&- \int_{X}
\left(\left( \Box + 1 \right)^{-1}(A\cdot \ol A)\right)
\cdot(\nu \cdot \ol\nu)\; g\, dV\nonumber\\
&& \quad - \int_{X} \left(\left( \Box + 1 \right)^{-1} (A\wedge\ol \nu)\right)
\cdot (\ol A \wedge \nu)\; g\, dV \\
&& \quad -  \int_{X} \left(\left( \Box - 1 \right)^{-1}
(A\wedge \nu)\right)\cdot (\ol A\wedge\ol \nu)\; g\, dV.
\nonumber
\end{eqnarray}
The only possible positive contribution arises from
$$
\int_{X}H(A\wedge \nu) H(\ol A \wedge \ol\nu)\, g\, dV.
$$
\end{theorem}
\begin{corollary}
The curvature can be estimated in terms of norms of harmonic projections.
\begin{equation}\label{eq:mainest}
R(A,\ol A, \nu,\ol \nu)\leq-  \|H (A\wedge\ol\nu)\|^2  + \|H(A\wedge \nu)\|^2.
\end{equation}
(Here $A\wedge \nu \in \cA^{(0,p+1)}(\cX_s,\Lambda^{p+1}\cT_{\cX_s})$ and $\ol A \we \nu \in \cA^{(0,p-1)}(\cX_s,\Lambda^{p-1}\cT_{\cX_s})$.)
\end{corollary}

\subsection*{Curvature for holomorphic families over curves}
When considering holomorphic families over (smooth analytic pieces of) curves $C$, we are dealing with holomorphic maps $ a :C\to S$, where $S$ is the base space of a universal deformation. We are looking at families  of the form $f_C: {\cX}_C \to C$, where $\cX_C = \cX\times_S C$. We suppose that $a$ has values in the locally free locus of $R^pf_*\Lambda^p\cT_{\cX/S}$.

Now the \ks map of order $p$ for such a curve is
$$
\rho^{(p)}: \cT^{\otimes p}_C \to R^pf_{C*}\Lambda^p\cT_{\cX_C/C} = a^*\!\!\left( R^pf_* \Lambda^p\cT_{\cX/S}\right).
$$
The \ks form of order $p$ induces a (semi-) norm $h_p$ on $\cT^{\otimes p}_C$.
\begin{lemma}[{\cite{sch-preprint10}}]\label{le:RAp}
  Let $s$ be a holomorphic coordinate on $C$. Then, identifying $\pt/\pt s$ with a harmonic form $A$ we have
  $$
  -\frac{\pt^2 \log \|A^p\|^2}{\pt s \ol{\pt s}}=-\frac{\pt^2 \log h_p}{\pt s \ol{\pt s}}\leq R(A, \ol A, A^p, \ol{A^p})/\|A^p\|^2,
  $$
  if $A^p\neq 0$.
\end{lemma}
\begin{proof}
  We set $\nu=A^p$, apply \eqref{eq:curvgendual}, and consider the second fundamental form.
\end{proof}

\begin{lemma}[{\cite[Lemma 13 (ii)]{to-yeung}}]
Let $\mu \in \cA^{(0,p-1)}(\Lambda^{p-1}\cT_{\cX_s})$ be harmonic, $A$ and $\nu$ as above. Then
\begin{equation}\label{eq:CauSch}
\|H(\ol A\we \nu)\|^2\geq |(H(A\we \mu),\nu)|^2/ \|\mu\|^2.
\end{equation}

\end{lemma}
\medskip

\begin{proof}
The claim follows immediately from Lemma~\ref{le:dualdual} and the Cauchy-Schwarz inequality.
\end{proof}

\begin{definition}
  A (semi-)norm $G_p=\|\mbox{\textvisiblespace}\|_p$ on $\cT_C$ is defined by
  $$
  \|A\|_p =\|A^p\|^{1/p}
  $$
  for any $p$.
\end{definition}
These norms are only continuous, where the $A^p$ vanish, and differentiable elsewhere. By Lemma~\ref{le:RAp}, the curvatures of the induced Finsler pseudo-metrics for a tangent vector $\pt/\pt s$ corresponding to the tensor $A$ satisfy
$$
K_p= - \frac{\pt^2 \log(\|A\|^2_p)}{\pt s \ol{\pt s}}/  \|A\|^2_p\leq  R(A,\ol A, A^p, \ol{A^p})/p\|A\|^{2+2p}_p .
$$
The estimate \cite[(60) Lemma 7]{sch-inv} can be improved using the Cauchy-Schwarz inequality from above setting $\nu=A^p$ and $\mu = A^{p-1}$ for $p>1$.

For $p>1$, our Theorem~\ref{th:curvgendual} implies
$$
R(A,\ol A, A^p, \ol{A^p}) \leq -\frac{\|A^p\|^4}{\|A^{p-1}\|^2} + \|A^{p+1} \|^2,
$$
whereas for $p=1$ the estimate already follows from  \cite{sch:curv}:
$$
R(A,\ol A, A, \ol A) \leq -\|H(A\we\ol A)\|^2 + \|H(A\we A)\|^2 ,
$$
and the constant $H(A\we \ol A)$ equals
$$
H(A\we \ol A)=\int_{\cX_s} (A \we \ol A) \;g\, dV \Big/ \mathrm{vol}(\cX_s)
$$
where the wedge product $A\we \ol A$ is given by \eqref{eq:we2}.
$$
R(A,\ol A,A,\ol A)/\|A\|^2\leq -c \cdot\|A\|^2 + \frac{\|A^2\|^2}{\|A\|^2}.
$$
Using the extra term in Theorem~\ref{th:curvgendual} we can take $c=-2/{\mathrm vol}(X)$. (Note that the volume of the fibers is constant an a polarized family.)
Hence
\begin{equation}
K_1\leq -c + \frac{\|A^2\|^2}{\|A\|^4}= -c +\frac{\|A\|^4_2}{\|A\|^4}.
\end{equation}
We follow our arguments from \cite{sch-inv} and obtain the estimate
  \begin{equation}
    K_p\leq- \frac{1}{p}\left(-\left(\frac{\|A\|^2_p}{\| A\|^2_{p-1}}\right)^{p-1} +\left( \frac{\|A\|^2_{p+1}}{\| A\|^2_{p}}\right)^{p+1}  \right)
  \end{equation}
which is valid whenever $A^p\neq 0$ and $A^{p-1}\neq 0$.
\begin{lemma}
Given a curve $C \to S$, assume that $G_j\not \equiv 0$ for $j=1,\ldots, q$, and $G_{q+1}\equiv 0$. Then on a complement of a discrete set of $C$
\begin{equation}\label{eq:Kpest}
K_1\leq -c + \frac{G^2_2}{G^2_1} \quad \text{ and } \quad   K_p\leq
\frac{1}{p}\left(- \frac{G^{p-1}_p}{G^{p-1}_{p-1}} + \frac{G^{p+1}_{p+1}}{G^{p+1}_{p}}\right) \text{ for } 1<p<q+1.
\end{equation}
\end{lemma}

\section{Finsler metrics}
We will use Demailly's version of Ahlfors Lemma to prove Kobayashi hyperbolicity of a complex space. In this way we can treat the locus, where the direct image sheaves $R^pf_*\Lambda^p_{\cX/S}\cT_{\cX/S}$ are {\em not locally free}.

\begin{theorem}[{Demailly \cite[3.2]{de:hyp}}]\label{th:dem}
Let $\gamma(t) = \ii \gamma_0(t) dt\we \ol{dt}$ be a hermitian metric on the unit disk $\Delta=\{|t|<0\}$, where $\log \gamma_0$ is a subharmonic function such that $\ii \pt \db \log \gamma_0 \geq A \gamma(t)$ in the sense of currents for some positive constant $A$. Then $\gamma$ can be compared with the Poincaré metric on $\Delta$ as follows:
$$
\gamma(t) \leq \frac{2}{A\cdot(1-|t|^2)^2}\ii dt\we \ol{dt}.
$$
\end{theorem}
The {\em curvature} of $\gamma$ (at $0$) is the infimum of all numbers $-A$, where $A$ satisfies the assumption of the Theorem.

Different notions of a Finsler metric are common. We do not assume the triangle inequality/convexity. Such metrics are also called {\em pseudo-metrics} (cf.\ \cite{kobook}).
\begin{definition}\label{de:fins}
Let $Z$ be a reduced complex space and let $T_cZ$ be the fiber bundle
consisting of the tangent cones of 1-jets. An upper semi-continuous
function
$$
F:T_cZ \to \R_{\geq 0}
$$
is called Finsler pseudo-metric (or pseudo-length function), if
$$
F(av)=|a|F(v) \text{ for all } a\in \C, v\in TZ.
$$
We require that the restriction to any (local) analytic, complex curve yields a possibly singular hermitian metric $\gamma(t) = \ii \gamma_0(t) dt\we \ol{dt}$
such that $\log \gamma_0(t)$ is subharmonic and not identically equal to $-\infty$.
\end{definition}

The triangle inequality on the fibers is not required for the definition
of the holomorphic (or holomorphic sectional) curvature in the direction of a given tangent vector $v$ at a point $p$: The holomorphic
curvature of a Finsler metric at a certain point $p$ in the direction of a
tangent vector $v$ is the supremum of the curvatures of the pull-back of
the given Finsler metric to a holomorphic disk through $p$ and tangent to
$v$ (cf.\ \cite{abate-patrizio}). (For a hermitian metric, the holomorphic
curvature is known to be equal to the holomorphic sectional curvature.)

Now Theorem~\ref{th:dem} together with the above set-up imply the hyperbolicity of spaces with negatively curved Finsler metrics:
\begin{proposition}\label{pr:ahlhyp}
Let $Z$ be a reduced complex space equipped with a Finsler metric $F$. Assume that the holomorphic sectional curvature of $F$ is bounded from above by a negative constant. Then $Z$ is hyperbolic in the sense of Kobayashi.
\end{proposition}
Given a Deligne-Mumford stack, the analytic methods also yield hyperbolicity of the stack or {\em hyperbolicity in the orbifold sense}, which in turn implies Brody hyperbolicity (in the orbifold sense).

\begin{proposition}\label{pr:stack}
Let $Z$ be a reduced complex space equipped with the analytic structure of an orbifold complex space. In particular $Z$ possesses a covering by  quotient spaces $S/\Gamma_S$, where $S$ is a reduced complex space and $\Gamma_S$ a finite group. Assume
\begin{itemize}
  \item[(i)] all spaces $S$ possess $\Gamma_S$-invariant Finsler metrics with the condition of Proposition~\ref{pr:ahlhyp}
  \item[(ii)] the data in $(i)$ are compatible with the orbifold structure of $Z$.
\end{itemize}
Then $Z$ is Kobayashi hyperbolic in the orbifold sense.
\end{proposition}
We will also apply this concept when dealing with the notion of hyperbolicity modulo a subspace.

\subsection*{Construction of a Finsler metric}
So far, we still restrict ourselves to the case, where all direct image sheaves are locally free.

The functions $G_p$ define Finsler (pseudo-)metrics. Above we set
$$
\|A\|^2_{G_p}=\|A\|^2_p=\|A^p\|^{2/p} = \left(\int_X \|H(A\we \ldots \we A)\|^2(z) g(z)\right)^{1/p} dV.
$$

\begin{lemma}[cf.\ {\cite[Lemma~4]{sch-preprint08}}]\label{le:convsum}
Let $C$ be a complex curve and $G_j$ a collection of pseudo-metrics of
bounded curvature, whose sum has no common zero. Then for $1\leq k\leq n$   the curvatures satisfy the following equation.
\begin{equation}\label{eq:curvestK}
K_{\sum_{j=1}^k G_j} \leq \sum_{j=1}^k \frac{G_j^2}{(\sum_{i=1}^k G_i)^2} K_{G_j}.
\end{equation}
\end{lemma}

Like in \cite{sch-preprint08} we construct a Finsler metric from the  $G_p$. Note that $G_1$ is the \wp norm, whereas the further quantities $G_p$ are semi-norms by definition, which are continuous and of class $\cinf$ on the complement of their zero-set.

\begin{definition}
Define a Finsler metric on the base spaces of universal deformations by
\begin{equation}\label{eq:convsum}
H= \sum^n_{p=1}p\, \alpha_p\, G_p
\end{equation}
with $\alpha_p>0$.
\end{definition}

Now we chose $k$ like in the above Lemma. So by \eqref{eq:curvestK}
\begin{equation}\label{eq:KH}
K_H \leq \frac{1}{H^2} \sum^k_{p=1} p\,\alpha_p \, G^2_p K_p.
\end{equation}
In principle, the Finsler metric depends upon $k$: We drop the further terms $G_p$ that are identically zero, and account for these in the formula for the  curvature estimate.

\begin{lemma}\label{le:Kest}
\begin{equation}\label{eq:curv}
K_H \leq \frac{1}{H^2} \Big(- c\, \alpha_1 G^2_1 - \sum^k_{p=2} \Big(\alpha_{p} \frac{G^{p+1}_p}{ G^{p-1}_{p-1}} - \alpha_{p-1} \frac{G^{p}_p}{G^{p-2}_{p-1}}  \Big)\Big)
\, .
\end{equation}
\end{lemma}
\begin{proof}
  The Lemma follows immediately from \eqref{eq:KH}, and  \eqref{eq:Kpest} with a shift of indices.
\end{proof}
Note that $\|A\|_p=0$ implies $\|A\|_{p+1}=0$. We first treat \eqref{eq:curv} in a {\em formal way}.
\begin{proposition}\label{pr:formest}
  There exist numbers $0<\alpha_1\leq \alpha_2,\ldots \leq \alpha_n$, and $K>0$ such that for all positive numbers $G_j$, $j=1,\ldots,k$ with $k\leq n$ and $H$ being given by \eqref{eq:convsum}, the right hand side of \eqref{eq:curv} is less or equal to $-K$.
\end{proposition}
The statement yields boundedness of the curvature, even if some of the $G_p$,
$p>1$ tend to zero.

It will be sufficient to treat the most general case $k=n$, because the coefficients $\alpha_j$ are chosen inductively. Details are in Section~\ref{se:formest}.

\section{Singular direct image sheaves -- \\ Application to universal deformations and moduli spaces}
In \cite{sch-inv} we explained, how to deal with the locus, where the direct image sheaves $R^pf_*\Lambda^p\cT_{\cX/S}$ are {\em not locally free}. The method is to consider transversal holomorphic curves in the following sense.

Let a holomorphic family of canonically polarized manifolds be defined over an analytic disk $C$. Then on the complement of a discrete set all direct images $R^pf_*\Lambda^p\cT_{\cX_C/C}$ are locally free. We assume that $0$ is the only such point. Denote by $H$ the above Finsler metric over $C\backslash \{0\}$.
\begin{lemma}\label{le:subhext}
The upper semi-continuous extension of $H= H(s)\ii ds \we \ol{ds}$ to $C$ gives rise to a \psh\ extension of $\log H$ such that
$$
\idb \log{H(s)}\geq K H(s)\ii ds \we \ol{ds}
$$
in the sense of currents, where $-K<0$ is the curvature of the Finsler metric $H$.
\end{lemma}
\begin{proof}
  We show that all $G_p$ are bounded from above. Let $A_s$ be the harmonic \ks form representing $\rho(\pt/\pt s|_s)$ for $s \neq 0$. On the other hand a differentiable trivialization of the family (near $s=0$) yields representatives $B_s$ for all $s$ including $0$ depending in a $\cinf$ way upon the parameter. Also the family of \ke metrics depends in a $\cinf$ way upon $s$. Now
  $$
  \|A^p_s\|^2 = \|H(A_s\we \ldots \we A_s)\|^2 \leq \|B_s \we\ldots\we B_s\|^2\leq \textit{const}.
  $$
  for all $s\neq 0$. So $\log H$ extends as a \psh\ function.
\end{proof}
  Now Theorem~\ref{th:dem} is applicable, and $H$ can be estimated from above by the Poincaré metric up to a factor of $1/K$.

\subsection*{Stratification of the moduli space}
We are carrying along the orbifold structure of the moduli space $\cM$ in the sense of Proposition~\ref{pr:stack}, and as described above all Finsler metrics are to be considered in the orbifold sense.

\begin{definition}
  Let $Z$ be a complex space and $W$ a closed subspace. Denote by $d_Z$ the Kobayashi pseudo-distance. Then $Z$ is called hyperbolic modulo $W$, if $d_Z(p,q)>0$ for all distinct points $p$ and $q$, unless both are contained in $W$.
\end{definition}

The moduli stack $\cM$ carries a natural stratification $ \cM_N\subsetneq \ldots \subsetneq \cM_1 \subsetneq \cM$, where all direct image sheaves $R^pf_*\Lambda^p\cT_{\cX/S}$ are locally free on all $\cM_j\backslash \cM_{j+1}$.

\begin{proposition}\label{pr:hypmod}
  All $\cM_j$ are hyperbolic modulo $\cM_{j+1}$.  Moreover, $\mathcal M_j\backslash \mathcal M_{j+1} \subset \mathcal M_j$ are hyperbolically embedded.
\end{proposition}
\begin{proof}
  By our construction, Lemma~\ref{le:Kest}, Proposition~\ref{pr:formest} the spaces $\cM_j\backslash \cM_{j+1}$ are hyperbolic, and Le-ma~\ref{le:subhext} together with Theorem~\ref{th:dem} imply the claim.
\end{proof}

\begin{theorem}
  The moduli space $\cM$ of canonically polarized compact manifolds is hyperbolic in the orbifold sense.
\end{theorem}
The {\em proof}\/ follows from Proposition~\ref{pr:hypmod} by induction.

\section{computation of the curvature}\label{se:proofcurv}
We present the idea of the proof.

Sections of \RP were denoted by letters like $\psi$.
\begin{gather*}
\psi|_{\cX_s} = \psi_{\alpha_1,\ldots,\alpha_p,\ol\beta_{p+1},\ldots,\ol\beta_n}
dz^{\alpha_1}\wedge \ldots \wedge dz^{\alpha_p}\wedge dz^{\ol\beta_{p+1}}\we
\ldots \we dz^{\ol\beta_n} \\
= \psi_{A_p\ol B_{n-p}} dz^{A_p}\we dz^{\ol B_{n-p}} \hspace{5cm}
\end{gather*}
where $A_p=(\alpha_1,\ldots,\alpha_p)$ and $\ol B_{n-p}=(\ol\beta_{p+1},
\ldots,\ol\beta_n)$. The form $\psi$ is taken according to Proposition~\ref{pr:harmsec}. The further component of $\psi$ is
$$
\psi_{\alpha_1,\ldots,\alpha_p,\ol\beta_{p+1},\ldots,\ol\beta_{n-1},\ol s}
dz^{\alpha_1}\wedge \ldots \wedge dz^{\alpha_p} \we dz^{\ol\beta_{p+1}}\we
\ldots \we dz^{\ol\beta_{n-1}}\we \ol{ds}.
$$
Now Proposition \ref{pr:harmsec} implies
\begin{equation}
\psi_{\alpha_1,\ldots,\alpha_p,\ol \beta_{p+1}, \ldots, \ol\beta_n |\ol s}
= \sum_{j=p+1}^n (-1)^{n-j}
\psi_{\alpha_1,\ldots,\alpha_p,\ol \beta_{p+1}, \ldots, \wh{\ol\beta}_j, \ldots, \ol\beta_n ,\ol s|\ol\beta_j }.
\end{equation}
Let
$$
H^{\ol \ell k}= \langle \psi^k,\psi^\ell\rangle.
$$
Derivatives with respect to a parameter of the base are needed. Let $s$ be a coordinate and $v$ the horizontal lift of $\pt/\pt s$ according to Lemma~\ref{le:horli}. The Lie derivative for tensors on the total space is denoted by $L_v$. Then
$$
\frac{\pt}{\pt s} H^{\ol\ell k} =  \langle L_v \psi^k, \psi^\ell  \rangle + \langle  \psi^k, L_\ol v  \psi^\ell \rangle
$$
because $L_v(g\/ dV)=0$.
Taking Lie derivatives is not type preserving -- in fact
$$
L_v\psi = L_v\psi' + L_v\psi'',
$$
where $(L_v\psi)'$ is of type $(p,n-p)$ and $(L_v\psi)''$ is of type
$(p-1,n-p+1)$.
A straightforward calculation shows the following relations.
\begin{lemma}
\begin{eqnarray}
  L_v\,\psi'' &=& A_s \cup \psi \label{eq:2} \\
  L_\ol v\,\psi'' &=& (-1)^p A_\ol s \cup \psi \label{eq:3}
\end{eqnarray}
\end{lemma}
Lie derivatives in conjugate directions yield $\ol\pt$-closed forms.
\begin{lemma}\label{le:lvpsi1}
\begin{equation}\label{eq:lvpsi1}
L_\ol v\,\psi'= (-1)^p\ol\pt (\ol v \cup \psi).
\end{equation}
\end{lemma}
Let $L_\cX$ be the multiplication by $\omega_\cX$. Then
$$
\ii [\ol\pt, \pt] = -  L_\cX,
$$
and the analogous formula holds fiberwise.
The following equation holds on  $\cA^{(p,q)}(\cK_{\cX_s})$.
\begin{equation}\label{eq:boxdbox}
\Box_\pt = \Box_\ol\pt +  (n-p-q) \cdot id.
\end{equation}
In particular, the harmonic forms $\psi \in \cA^{(p,n-p)}(\cK_{\cX_s})$
are also harmonic with respect to $\pt$.

We are in a position to compute the curvature tensor.
Because of \eqref{eq:lvpsi1}
$$
\langle \psi^k, (L_\ol v \psi^\ell)'  \rangle = 0
$$
holds for all $s\in S$ so that by Proposition~\ref{pr:harmsec}
$$
\frac{\pt}{\pt s} H^{\ol\ell k} =
\langle (L_v \psi^k)', \psi^\ell \rangle.
$$
When using normal coordinates (of the second kind) at a given fixed point $s_0$,
the condition $(\pt/\pt s)H^{\ol\ell
k}|_{s_0}=0$ for all $k,\ell$ means that for $s=s_0$ the harmonic
projection
\begin{equation}\label{eq:HLv}
H((L_v \psi^k)') =0
\end{equation}
vanishes for all $k$.

In order to calculate second order derivatives of the metric tensor in an effective way, the order of taking derivatives will best be changed. This requires Lie derivative $L_{[v,\ol v]}$:
\begin{equation}\label{eq:Lolvv}
\pt_\ol s\pt_s \langle \psi^k,\psi^\ell \rangle = \langle L_{[\ol
v,v]}\psi^k , \psi^\ell\rangle -\langle L_\ol v\psi^k,L_\ol
v\psi^\ell\rangle + \langle L_v\psi^k,L_v\psi^\ell\rangle
\end{equation}
\begin{lemma}\label{le:vvb}
Restricted to the fibers $\cX_s$ the following equations hold for $L_{[\ol
v, v]}$ applied to $\cK_{\cX/S}$-valued functions and
differential forms resp.
    \begin{eqnarray}\label{eq:vvb}
      L_{[\ol v, v]} &=&
      \big[-\varphi^{;\alpha}\pt_\alpha + \varphi^{;\ol\beta}\pt_{\ol\beta},\; \mbox{\textvisiblespace} \;\big]
      - \varphi \cdot id\\
      \label{eq:lvvpsi}
\langle L_{[\ol v, v]} \psi^k , \psi^\ell  \rangle = - \langle \varphi\, \psi^k, \psi^\ell\rangle
&= &-  \int_{\cX_s} (\Box + 1)^{-1}(A_s \cdot A_\ol s) \psi^k \psi^\ol\ell \, g \, dV .
    \end{eqnarray}
\end{lemma}

The following identities are necessary
\begin{eqnarray}
\ol\pt(L_v\psi^k)'&=&  \pt(A_s\cup \psi^k)\label{eq:0} \\
  \ol\pt^*(L_v\psi^k)'&=& 0 \label{eq:4} \\
  \pt^*(A_s\cup\psi^k) &=&0 \label{eq:5} \\
 \ol\pt^* (L_\ol v\psi^k)'&=&  \pt^* (A_\ol s \cup \psi^k) \label{eq:6}\\
  \ol\pt(L_\ol v\psi^k)'&=& 0 \label{eq:1} \\
  \pt(A_\ol s \cup \psi^k) &=&0 \label{eq:7}
\end{eqnarray}
Now the proof of the curvature formula \eqref{eq:curvgen} (using normal coordinates) proceeds as follows: We continue with \eqref{eq:Lolvv} and apply \eqref{eq:lvvpsi}. Let $G_\pt$
and $G_\ol\pt$ denote the Green's operators on the spaces of
differentiable $\cK_{\cX_s}$-valued $(p,q)$-forms on the fibers with
respect to $\Box_\pt$ and $\Box_\ol\pt$ resp. We know from
\eqref{eq:boxdbox} that for $p+q=n$ the Green's operators $G_\pt$
and $G_\ol\pt$ coincide.

Since the harmonic projection $H ((L_v\psi^k)')=0$ vanishes for $s=s_0$, we have
\begin{gather*}
(L_v\psi^k)'= G_\ol\pt \Box_\ol\pt (L_v\psi^k)'= G_\ol\pt \ol\pt^*\ol\pt
(L_v\psi^k)' = \ol\pt^*G_\ol\pt \pt(A_s\cup\psi^k)
\end{gather*}
by \eqref{eq:4} and \eqref{eq:0}.  The form $\ol\pt(L_v\psi^k)'=
\pt(A_s\cup \psi^k)$ is of type $(p,n-p+1)$ so that by
\eqref{eq:boxdbox}  on the space of such forms $G_\ol\pt=(\Box_\pt
+1)^{-1} $ holds.

Now
\begin{gather*}
\hspace{-10mm}
\langle (L_v\psi^k)', (L_v\psi^\ell)'\rangle =\langle \ol\pt^*G_\ol\pt
\pt(A_s\cup\psi^k), (L_v\psi^\ell)' \rangle = \langle G_\ol\pt
\pt(A_s\cup\psi^k), \pt (A_s \cup \psi^\ell) \rangle \\ = \langle (\Box_\pt
+1)^{-1} \pt (A_s\cup \psi^k), \pt (A_s\cup \psi^\ell)\rangle
= \langle \pt^* (\Box_\pt +1)^{-1} \pt (A_s\cup \psi^k), A_s\cup
\psi^\ell\rangle.
\end{gather*}

Because of \eqref{eq:5}
\begin{gather*}
\langle (L_v\psi^k)', (L_v\psi^\ell)'\rangle = \langle (\Box_\pt
+1)^{-1}\Box_\pt (A_s \cup \psi^k) , A_s\cup \psi^\ell\rangle \\
\strut\hspace{10mm}=
\langle A_s\cup \psi^k, A_s \cup\psi^\ell\rangle -1 \langle (\Box+1)^{-1}(A_s \cup \psi_k), A_s \cup \psi^\ell\rangle.
\end{gather*}
(For $(p-1, n-p+1)$-forms, we write $\Box=\Box_\pt=\Box_\ol\pt$.)

Altogether we have
\begin{equation}\label{eq:part2}
\langle L_v \psi^k , L_v \psi^\ell \rangle|_{s_0} = - 1
\int_{\cX_s} (\Box +1)^{-1}(A_s\cup\psi^k)\cdot (A_\ol s\cup \psi^\ol\ell)\, g\, dV.
\end{equation}
We still need to compute $\langle L_\ol v\psi^k, L_\ol v \psi^\ell
\rangle$:

By equation \eqref{eq:3} we have  $(\langle L_\ol v\psi^k)'', (L_\ol v
\psi^\ell)'' \rangle = \langle A_\ol s \cup \psi^k   , A_\ol s \cup
\psi^\ell \rangle$. Now Lemma~\ref{le:lvpsi1} implies that the harmonic
projections of the $(L_\ol v \psi^k)'$ vanish for all parameters $s$. So
\begin{gather*}
\langle (L_\ol v \psi^k)',  (L_\ol v \psi^\ell)'\rangle = \langle G_{\ol\pt}
\Box_{\ol\pt} (L_\ol v \psi^k)',  (L_\ol v \psi^\ell)'\rangle
\underset{\eqref{eq:1}}{=} \langle
(G_\ol\pt \ol\pt {\ol\pt}^* L_\ol v \psi^k)', (L_\ol v \psi^\ell)'\rangle
= \\ \hspace{30mm} \langle (G_\ol\pt {\ol\pt}^* L_\ol v \psi^k)', \ol\pt^* (L_\ol v \psi^\ell)'\rangle
\underset{\eqref{eq:6}}{=} \langle
G_{\ol\pt} \pt^*(A_\ol s \cup \psi^k), \pt^*(A_\ol s\cup \psi^\ell) \rangle.
\end{gather*}
Now the $(p+1, n-p)$-form $\ol\pt^*(L_\ol v\psi^k)'=  \pt^* (A_\ol s \cup
\psi^k)$ is orthogonal to both of the spaces of $\ol\pt$- and $\pt$-harmonic
forms. On these, by \eqref{eq:boxdbox} we have
$$
\Box_\ol\pt=\Box_\pt - \cdot id.
$$
We see that all eigenvalues of $\Box_\pt$ are larger or equal to $1$ for
$(p,n-p-1)$-forms.

We have  from \eqref{eq:3}
$$
(\langle L_\ol v\psi^k)'', (L_\ol v
\psi^\ell)'' \rangle = \langle A_\ol s \cup \psi^k   , A_\ol s \cup
\psi^\ell \rangle.
$$
Now Lemma~\ref{le:lvpsi1} implies that the harmonic
projections of the $(L_\ol v \psi^k)'$ terms vanish for all parameters $s$. So
\begin{gather*}
\langle (L_{\ol v} \psi^k)',  (L_{\ol v} \psi^\ell)'\rangle = \langle G_{\ol\pt}
\Box_{\ol\pt} (L_{\ol v} \psi^k)',  (L_\ol v \psi^\ell)'\rangle \hspace{0cm}
\underset{\eqref{eq:1}}{=} \langle
(G_\ol\pt \ol\pt \ol\pt^* L_\ol v \psi^k)', (L_\ol v \psi^\ell)'\rangle \\ \hspace{5cm}
= \langle
(G_{\ol\pt} \ol\pt^* L_\ol v \psi^k)', \ol\pt^* (L_\ol v \psi^\ell)'\rangle
\underset{\eqref{eq:6}}{=}  \langle G_{\ol\pt} \pt^*(A_\ol s \cup \psi^k), \pt^*(A_{\ol s}\cup \psi^\ell) \rangle.
\end{gather*}
Now the $(p+1, n-p)$-form $\ol\pt^*(L_\ol v\psi^k)'=  \pt^* (A_\ol s \cup
\psi^k)$ is orthogonal to both the spaces of $\ol\pt$- and $\pt$-harmonic
forms. On these,  we have by \eqref{eq:boxdbox} that $\Box_\ol\pt=\Box_\pt - id$. We see that all eigenvalues of $\Box_\pt$ are larger or equal to $1$ for
$(p,n-p-1)$-forms.

One can see that in the eigenfunction decomposition of $A_\ol s\cup \psi^k$ only eigenfunctions for eigenvalues $0$ or larger that $1$ occur so that $(\Box - 1)^{-1}(A_\ol s\cup \psi^k)$ exists.

Now $ G_\ol\pt  \pt^*(A_\ol s \cup \psi^k) = (\Box_\pt -1)^{-1} \pt^*(A_\ol s \cup \psi^k)$, and \eqref{eq:7} implies
\begin{gather*}
\langle (L_\ol v \psi^k)',  (L_\ol v \psi^\ell)'\rangle = \langle
(\Box_\pt -1)^{-1} \Box_\pt (A_\ol s \cup \psi^k) ,A_\ol s \cup \psi^\ell
\rangle \\ \hspace{15mm} = \langle A_\ol s \cup \psi^k ,A_\ol s \cup \psi^\ell \rangle +
\langle (\Box_\pt -1)^{-1} (A_\ol s \cup \psi^k) ,A_\ol s \cup \psi^\ell
\rangle.
\end{gather*}
Equation \eqref{eq:3} yields the final equation (again with $\Box_\ol\pt =
\Box_\pt = \Box$ for $(p+1,n-p-1)$-forms)
\begin{equation}\label{eq:part3}
\langle L_\ol v \psi^k,  L_\ol v \psi^\ell\rangle =  \int_{\cX_s} (\Box - 1)^{-1}( A_\ol s \cup \psi^k)
\cdot (A_s \cup \psi^\ol\ell) \, g\, dV.
\end{equation}
The statement of Theorem~\ref{th:curvgen} follows from \eqref{eq:lvvpsi}, \eqref{eq:part2},
\eqref{eq:Lolvv}, and \eqref{eq:part3}. \qed

\section{Proof of Proposition~\ref{pr:formest}}\label{se:formest}
The proof is very elementary and given for the sake of completeness. In view of \eqref{eq:curv} it is obvious to use the following estimate.
\begin{lemma}\label{le:firstineq}
Let $p\in \mathbb N$. Then
$$
f(x)=x^{p+1} - x^p - x^2/2 + 1/2 \geq 0
$$
for all $ x \geq 0 $ .
\end{lemma}
\begin{lemma}\label{le:secineq}
Let $\alpha_p>0$ for $p=1,\ldots, n$. Then for all $x_p\geq 0$
\begin{eqnarray}
&&\sum^n_{p=2} (\alpha_p x^{p+1}_p - \alpha_{p-1}x^p_p) x^2_{p-1}\cdot\ldots\cdot x^2_1\nonumber\\
&&\strut\qquad\geq
 \frac{1}{2}\left(
- \frac{\alpha^3_1}{\alpha^2_2} x^2_1 +  \frac{\alpha^{n-1}_{n-1}}{\alpha^{n-2}_n} x^2_n \cdot\ldots\cdot x^2_1 +  \sum^{n-1}_{p=2} \Big(\frac{\alpha^{p-1}_{p-1}}{\alpha^{p-2}_p} -\frac{\alpha^{p+2}_p}{\alpha^{p+1}_{p+1}} \Big)x^2_p\cdot\ldots\cdot x^2_1\right)
\end{eqnarray}
holds.
\end{lemma}
\medskip
\begin{proof}
Use
$
\frac{\alpha^{p+1}_{p-1}}{\alpha^p_{p}}f\big(\frac{\alpha_{p}}{\alpha_{p-1}}x_p\big)\geq 0
$
from  Lemma~\ref{le:firstineq} and take a sum.
\end{proof}
\medskip

Now we set $x_p=G_p/G_{p-1}$ with and apply Lemma~\ref{le:secineq}  to the result of Lemma~\ref{le:Kest}.
\begin{equation}\label{eq:KHH}
H^2 K_H \leq \left(-c  + \frac{1}{2}\frac{\alpha^2_1}{\alpha^2_2} \right)\alpha_1 G^2_1 + \frac{1}{2} \sum^{n-1}_{p=2}\left(\frac{\alpha^{p+2}_p}{\alpha^{p+1}_{p+1}} - \frac{\alpha^{p-1}_{p-1}}{\alpha^{p-2}_{p}}  \right)G^2_p -\frac{1}{2} \frac{\alpha^{n-1}_{n-1}}{\alpha^{n-2}{n}}G^2_n =: -\sum^n_{p=1} \gamma_p G^2_p.
\end{equation}
\begin{lemma}
  There exists numbers $0<\alpha_< \ldots< \alpha_n$ such that all coefficients $\gamma_p$ are positive.
\end{lemma}
\begin{proof}
We set $\alpha_1=1$ and chose the further values $\alpha_p$ inductively.
\end{proof}
Let $\gamma= \min(\gamma_p)$ in \eqref{eq:KHH}. The the right-hand side of the inequality can be estimated from above by $K H^2 =-(\gamma/n) H^2$. This proves Proposition~\ref{pr:formest}. \qed

\end{document}